\documentclass[12pt]{amsart}
\usepackage{epsfig}
\usepackage{amsmath}
\usepackage{amssymb}
\usepackage{amscd}
\usepackage{color}

\newtheorem{thm}{Theorem}[section]
\newtheorem{cor}[thm]{Corollary}
\newtheorem{lem}[thm]{Lemma}
\newtheorem{pro}[thm]{Proposition}

\theoremstyle{definition}
\newtheorem{definition}[thm]{Definition}
\newtheorem{remark}{Remark}

\numberwithin{equation}{section}

\begin{document}

\title[Fundamental domain of invariant sets]{Fundamental domain of invariant sets and applications}

\author[Pengfei Zhang]{Pengfei Zhang}
\email{pfzh311@gmail.com}

\address{School of Mathematical Sciences, Peking University,
Beijing 100871, China}



\thanks{The work is partially supported by NNSF project (11071231).}

\begin{abstract}
Let $X$ be a compact metric space and $f:X\to X$ a homeomorphism on
$X$. We construct a fundamental domain for the set with finite peaks
for each cocycle induced by $\phi\in C(X,\mathbb{R})$. In particular
we prove that if a partially hyperbolic diffeomorphism is
accessible, then either the set with finite peaks for the Jacobian
cocycle is of full volume, or the set of transitive points is of
positive volume.
\end{abstract}

\maketitle

\section{Introduction}

In this paper we give a construction of fundament domains for some
general subsets. More precisely let $X$ be a compact metric space,
$f:X\to X$ a homeomorphism and $E$ be an $f$-invariant set. If there
is an $f$-invariant Borel map $s:E\to E$ such that
$s(x)\in\mathcal{O}(f,x)$, then the image of $s$ is called a {\it
fundamental domain} of $E$. Take the {\it North and South Poles Map}
$f:S^2\to S^2$ for example: the set $E=S^2\backslash\{N,S\}$ is
$f$-invariant and $B(S,r)\backslash fB(S,r)$ is a fundamental domain
of $E$ (for $r<1$).

In general let $\phi\in C(X,\mathbb{R})$ be a continuous function.
This induces an {\it additive cocycle} $\{\phi_n:n\in\mathbb{Z}\}$
over $(X,f)$ which is given by
\[\phi_n(x)=\begin{cases}\phi(x)+\cdots+\phi(f^{n-1}x), & n\ge0;\\
-\phi(f^{n}x)-\cdots-\phi(f^{-1}x),& n<0.\end{cases}\] Let
$\Phi_f(x)=\sup_{n\in\mathbb{Z}}\phi_n(x)$ be the peak value at $x$.
Then the cocycle $\{\phi_n\}$ is said to have {\it finite peaks} at
a point $x\in X$, if $\{n\in\mathbb{Z}:\phi_n(x)=\Phi_f(x)\}$ is
nonempty and finite. Denote by $H(f,\phi)$ the set of points with
finite peaks.

Note that for some $\phi$ related to the dynamics, $H(f,\phi)$ can
be {\it quite large} with respect to natural measures, see Section
\ref{smooth} and \cite{GO}. Also $H(f,\phi)$ could be {\it large} in
the sense of entropy, see Remark \ref{large} and Section
\ref{entropy}. We prove that there always exists a fundamental
domain for this set:

\vskip0.3cm

\noindent\textbf{Theorem A.} {\it Let $f:X\to X$ be a homeomorphism
and $\phi\in C(X,\mathbb{R})$. Then there is an $f$-invariant Borel
section $\pi:H(f,\phi)\to H(f,\phi)$. Equivalently, the image of
$\pi$ is a fundamental domain of $H(f,\phi)$. }

We also give some applications of our construction. Let $f:M\to M$
be a transitive diffeomorphism and $\mathrm{Tran}_f$ be the set of
transitive points of $f$. It is well known that $\mathrm{Tran}_f$ is
a residual subset (hence topologically large). But a residual subset
could be measure-theoretically meagre (take the set of Liouville
numbers for example).

There are some classical results about the measure-theoretical
largeness of the transitive set. For example let $f:M\to M$ be a
$C^2$ transitive Anosov diffeomorphism. Sina\v{\i} proved in
\cite{Si} that there exists a unique Gibbs measure $\mu_+$ with
respect to $f$ whose basin $B(\mu_+,f)$ is of full volume:
$m(B(\mu_+,f))=1$. In particular $B(\mu_+,f)\subset \mathrm{Tran}_f$
since the support $\mathrm{supp}(\mu_+)=M$. So
$m(\mathrm{Tran}_f)=1$ for every $C^2$ transitive Anosov
diffeomorphism. See \cite{AB,Z} for recent results about the measure
of transitive sets for general systems. We get similar estimates of
$\mathrm{Tran}_f$ for accessible partially hyperbolic systems.
Namely let $J_f(x)$ be the Jacobian of $f$ with respect to the
Riemannian metric which induces $m$.

\vskip0.3cm

\noindent\textbf{Theorem B.} {\it Let $f:M\to M$ be a $C^2$
partially hyperbolic diffeomorphism. If $f$ is essentially
accessible and $m(H(f,\log J_f))<1$, then $f$ is transitive and
$m(\mathrm{Tran}_f)\ge 1-m(H(f,\log J_f))>0$. }

Once again let $f$ be $C^2$ transitive Anosov diffeomorphism,
$\mu_+$ (respectively, $\mu_-$) be the unique Gibbs measure with
respect to $f$ (respectively, $f^{-1}$). Denote the common measure
by $\mu$ if $\mu_+=\mu_-$. Following dichotomy is proved by Gurevich
and Oseledets \cite{GO}:
\begin{itemize}
\item either $\mu_+\neq\mu_-$: then $(m,f)$ is completely
dissipative;

\item or $\mu_+=\mu_-$: then $\mu$ is equivalent to $m$ and
$(m,f)$ is ergodic.
\end{itemize}
We also get a partial generalization of Gurevich and Oseledets
dichotomy to center bunched, essentially accessible partially
hyperbolic systems.

\vskip0.3cm

\noindent\textbf{Corollary C.} {\it Let $f:M\to M$ be a $C^2$
essentially accessible, center bunched partially hyperbolic
diffeomorphism.
\begin{enumerate}
\item Either $m(H(f,\log J_f))=1$: then $(m,f)$ is completely
dissipative;

\item or $m(H(f,\log J_f))<1$: then  $m(H(f,\log J_f))=0$ and $(m,f)$ is ergodic.
\end{enumerate}
}

Finally we give a proof of the entropy largeness of the Birkhoff
heteroclinic set $H_f(\mu,\nu)=B(\mu,f)\cap B(\nu,f^{-1})$ of two
$f$-invariant measures $\mu,\nu$, which may be of interest in its
own right. We are grateful to G. Liao for pointing out the asymmetry
of Bowen dimensional entropy $h_B(f,\cdot)$.

\vskip0.3cm

\noindent\textbf{Proposition D.} {\it Let $f:M\to M$ be a transitive
Anosov diffeomorphism. Then for all $f$-invariant measures
$\mu,\nu$, the entropy of $H_f(\mu,\nu)$ satisfies
$h_B(f,H_f(\mu,\nu))=h_\mu(f)$ and
$h_B(f^{-1},H_f(\mu,\nu))=h_\nu(f)$.

In particular if $h_\mu(f)\neq h_\nu(f)$, then
$h_B(f,H_f(\mu,\nu))\neq h_B(f^{-1},H_f(\mu,\nu))$. }

\section*{Acknowledgments}

We thank Jon Aaronson, Shaobo Gan, Gang Liao, Lan Wen, Shuyun
(Conan) Wu and Zhihong Xia for useful discussions. The author
benefited a lot from Shaobo Gan and Gang Liao's suggestions. We
thank Shuyun (Conan) Wu for sending me a copy of the paper
\cite{GO}.

\section{Fundamental domain of some invariant subsets}

Let $X$ be a compact metric space, $f:X\to X$ a homeomorphism and
$\phi\in C(X,\mathbb{R})$ be a continuous function. The induced
cocycle $\{\phi_n:n\in\mathbb{Z}\}$ over $(X,f)$ is given by
\[\phi_n(x)=\begin{cases}\phi(x)+\cdots+\phi(f^{n-1}x), & n\ge0;\\
-\phi(f^{n}x)-\cdots-\phi(f^{-1}x),& n<0.\end{cases}\] In particular
$\phi_0(x)\equiv0$, $\phi_{n+k}(x)=\phi_n(x)+\phi_k(f^nx)$ for all
$n,k\in\mathbb{Z}$ and $x\in X$.

\begin{definition}
Let $f:X\to X$ be a homeomorphism, $\phi\in C(X,\mathbb{R})$ and
$\Phi_f(x)=\sup_{n\in\mathbb{Z}}\phi_n(x)$. The cocycle $\{\phi_n\}$
is said to have {\it finite peaks} at a point $x\in X$ if
$\{n\in\mathbb{Z}:\phi_n(x)=\Phi_f(x)\}$ is nonempty and finite.
Denote by $H(f,\phi)$ the set of points with finite peaks.
\end{definition}

It is easy to see the set $H(f,\phi)$ is a Borel subset. And
$H(f,\phi)=\emptyset$ if $\phi$ is constant. In the following we
assume that $H(f,\phi)\neq\emptyset$.

\begin{definition}
The function $n_f$ of {\it last peak time} and the function $\pi$ of
{\it last peak position} on $H(f,\phi)$ are defined as:
\begin{equation}\label{npi}
n_f(x)=\sup\{n\in\mathbb{Z}:\phi_n(x)=\Phi_f(x)\},\text{ and }
\pi(x)=f^{n_f(x)}x.
\end{equation}
\end{definition}

Let $H_N=\{x\in X:\phi_{n}(x)<\Phi_f(x)\text{ for all }n\text{ with
}|n|> N \}$ for each $N\ge1$. It is clear that
$H(f,\phi)=\bigcup_{N\ge1}H_N$ and $\Phi_f(x)=\max_{|n|\le
N}\phi_n(x)$ on $H_N$. So the set $H_N$ is a $G_\delta$-subset and
the function $\Phi_f$ is continuous on $H_N$. Also note that for
each $x\in H_N$, $\phi_n(x)<\Phi_f(x)$ for all $|n|> N$ and hence
$|n_f(x)|\le N$. Moreover we have

\begin{lem}\label{borel}
Let $f$ be a homeomorphism on $X$, $\phi\in C(X,\mathbb{R})$ and
$H_N$ given as above. Then $n_f$ is upper semi-continuous on $H_N$
and the restriction $\pi|_{H_N}$ is a Borel map.
\end{lem}
\begin{proof}
(1). Let $x\in H_N$ and $x_k\in H_N\to x$. Note that $|n_f(x_k)|\le
N$ for all $k\ge1$. Passing to a subsequence if necessary, we assume
$n_f(x_k)=\hat{n}$ for all $k\ge1$. Now we claim $n_f(x)\ge \hat{n}$
and hence $n_f$ is upper semi-continuous on $H_N$.

If this were not true, then $n_f(x)<\hat{n}$ and hence $\Phi_f(x)>
\phi_{\hat{n}}(x)$. Since $\Phi_f$ is continuous on $H_N$, there
exists $\delta>0$ such that $\Phi_f(y)> \phi_{\hat{n}}(y)$ for all
$y\in H_N\cap B(x,\delta)$. In particular $\Phi_f(x_k)>
\phi_{\hat{n}}(x_k)$ for all $k$ large, which contradict the
assumption $\hat{n}=n_f(x_k)$. Therefore
$\limsup_{k\to\infty}n_f(x_k)\le n_f(x)$. This finishes the proof of
the claim and hence the first conclusion.

(2). Let $H(n)=\{x\in H_N:n_f(x)=n\}$. Clearly $H(n)$ is a Borel set
and $\bigsqcup_{|n|\le N}H(n)= H_N$. Then $\pi|_{H(n)}=f^n|_{H(n)}$
is a Borel map for each $|n|\le N$. So $\pi|_{H_N}$ is also Borel.
\end{proof}

\begin{thm}\label{fund}
Let $f:X\to X$ be a homeomorphism and $\phi\in C(X,\mathbb{R})$.
Then the map $\pi$ is an $f$-invariant Borel section on $H(f,\phi)$.
Equivalently, the image of $\pi$ is a fundamental domain of
$H(f,\phi)$.
\end{thm}
\begin{proof}
Let $k\in\mathbb{Z}$. Since $\phi_{n+k}=\phi_k+\phi_n\circ f^k$ for
all $n\in\mathbb{Z}$, we see that
\begin{equation}\label{shift}
\Phi_f(x)=\sup_{n\in\mathbb{Z}}\phi_{n+k}(x)
=\sup_{n\in\mathbb{Z}}\phi_{n}(f^kx)+\phi_k(x)
=\Phi_f(f^kx)+\phi_k(x).
\end{equation}
For \eqref{shift} we see that for every $x\in H(f,\phi)$:
\begin{itemize}
\item $\phi_{n+k}(x)=\Phi_f(x)$ if and only if
$\phi_n(f^kx)=\Phi_f(f^kx)$;

\item moreover $n_f(x)=n_f(f^kx)+k$.
\end{itemize}
Then we have
\[\pi(f^kx)=f^{n_f(f^kx)}(f^kx)=f^{n_f(f^kx)+k}(x)=f^{n_f(x)}(x)=\pi(x).\]
Thus $\pi$ is $f$-invariant and $\pi(x)\in\mathcal{O}_f(x)$. By
Lemma \ref{borel} we see that $\pi:H(f,\phi)\to H(f,\phi)$ is an
$f$-invariant Borel section and its image $W=\pi(H(f,\phi))$ is a
fundamental domain of $H(f,\phi)$. This completes the proof.
\end{proof}

\section{Applications: topological systems}\label{topo}

Let $f:X\to X$ be a homeomorphism and $\mathcal{M}(f)$ the
$f$-invariant probability measures. Let
$\nu_{x,n}=\frac{1}{n}\sum_{k=0}^{n-1}\delta_{f^kx}$ be the Birkhoff
average along the orbit segment $\{x,\cdots,f^{n-1}x\}$. Then the
basin $B(\mu,f)$ of $\mu$ with respect to $f$ is defined as
$B(\mu,f)=\{x\in X:\nu_{x,n}\to\mu \text{ as }n\to+\infty\}$, which
can be viewed as the {\it Birkhoff stable set} of the measure $\mu$.
In this spirit we give the following:

\begin{definition}
Let $\mu,\nu\in\mathcal{M}(f)$ be two distinct invariant measures.
The set of the {\it Birkhoff heteroclinic}\footnote{If we denote
$B^s(\mu,f)=B(\mu,f)$ and $B^u(\nu,f)=B(\nu,f^{-1})$, then
$H_f(\mu,\nu)=B^s(\mu,f)\cap B^u(\nu,f)$.} points of the pair
($\mu,\nu$), denoted by $H_f(\mu,\nu)$, is defined as
$H_f(\mu,\nu)=B(\mu,f)\cap B(\nu,f^{-1})$.
\end{definition}
\begin{remark}\label{large}
Let $f:M\to M$ be a transitive Anosov diffeomorphism. By {\it Limit
Shadowing Property} we see that $B(\mu,f)\neq\emptyset$ for all
invariant measure $\mu$, so is $B(\nu,f^{-1})$. Also note that
$B(\mu,f)$ is saturated by stable manifolds and $B(\nu,f^{-1})$ is
saturated by unstable manifolds. Therefore $H_f(\mu,\nu)$ is dense
for all invariant measures $(\mu,\nu)$. In fact
$h_B(f,H_f(\mu,\nu))=h_\mu(f)$ where $h_B(f,E)$ be Bowen's
dimensional entropy on noncompact sets (see Section \ref{entropy}).
So there are many invariant pairs with large heteroclinic sets
$H_f(\mu,\nu)$.
\end{remark}

The following theorem provides a fundamental domain of the Birkhoff
heteroclinic set:
\begin{thm}
Let $(X,f)$ be given as above, and $\mu,\nu\in\mathcal{M}(f)$ with
$H_f(\mu,\nu)\neq\emptyset$. Then there exists an $f$-invariant,
Borel section $s$ on $H_f(\mu,\nu)$ and its image is a fundamental
domain of $H_f(\mu,\nu)$.
\end{thm}
\begin{proof}
Since $\mu\neq \nu$, there exists a continuous function
$\phi:X\to\mathbb{R}$ such that $\int_X\phi\, d\mu\neq \int_X\phi\,
d\nu$. Replacing $\phi$ by $a\phi+b$ if necessary, we assume that
\[\int_X\phi\, d\mu=-1<0<1=\int_X\phi\, d\nu.\]

For each point $x\in H_f(\mu,\nu)=B(\mu,f)\cap B(\nu,f^{-1})$ we see
\begin{align*}
\lim_{n\to+\infty}\frac{1}{n}\phi_n(x)&
=\lim_{n\to+\infty}\frac{1}{n}\sum_{k=0}^{n-1}\phi(f^kx)=\int_X\phi\, d\mu=-1,\\
\lim_{n\to+\infty}\frac{1}{n}\phi_{-n}(x)&
=\lim_{n\to+\infty}\frac{1}{n}\sum_{k=1}^{n}-\phi(f^{-k}x)=-\int_X\phi\,
d\nu=-1.
\end{align*}
Therefore $\phi_n(x)<0=\phi_0(x)$ for all $n$ with $|n|$ large. So
$x\in H(f,\phi)$. Thus $H_f(\mu,\nu)\subset H(f,\phi)$ and the
restriction of $\pi$ (given by \eqref{npi}) to $H_f(\mu,\nu)$
provides the $f$-invariant section by Theorem \ref{fund}. This
completes the proof.
\end{proof}
Similarly we can define the {\it Birkhoff homoclinic set}
$H_f(\mu)=B(\mu,f)\cap B(\mu,f^{-1})$. We first note that there is
an obstruction for the existence of fundamental domain of Birkhoff
homoclinic sets for ergodic measures:
\begin{proof}
Let $\mu$ be an ergodic measure. Then $\mu(B(\mu,f))=\mu(
B(\mu,f^{-1}))=1$ by Birkhoff ergodic theorem. So $\mu(H_f(\mu))=1$.
If there were some fundamental domain $W$ of $H_f(\mu)$, then either
$\mu(W)=0$ (forces $\mu(H_f(\mu))=0$) or $\mu(W)>0$ (forces
$\mu(H_f(\mu))=\infty$), contradicts $\mu(H_f(\mu))=1$. So there do
not exist any fundamental domain of $H_f(\mu)$.
\end{proof}

Note that $\mu(B(\mu,f))=0$ for general
$\mu\in\mathcal{M}(f)\backslash\mathcal{M}^e(f)$. So the obstruction
no longer exists if $\mu$ is not ergodic. Moreover the basin
$H_f(\mu)$ could be large in the sense of entropy. In fact
$h_{B}(f,H_f(\mu))=h_\mu(f)$ if $f$ is a transitive Anosov
diffeomorphism (by Remark \ref{large}). We don't know if one can
find a fundamental domain of $H_f(\mu)$ for these measures.

\vskip.4cm

Now we give a simple corollary which will be used in next section.
Let $X$ be a compact metric space and $\mu$ be a probability measure
on $X$. Let $f:X\to X$ be a homeomorphism (may not preserves $\mu$).
\begin{definition}
Let $W$ be a measurable subset of positive $\mu$-measure. Then $E$
is said to be {\it wandering} with respect to $(\mu,f)$ if $f^nW$,
$n\in\mathbb{Z}$ are mutually disjoint. The {\it dissipative part}
$D_f$ of the system $(X,\mu,f)$ is the measurable union of the
collection of measurable wandering sets with respect to $(\mu,f)$.
The set $C_f=X\backslash D_f$ is called the {\it conservative part}
of $(\mu,f)$. The partition $\{C_f,D_f\}$ is called the {\it Hopf
decomposition} of $(\mu,f)$.
\end{definition}
Note that $D_f=\emptyset$ if every measurable set of positive
measure is not wandering. See \cite{H,H1,A,K} for more details.

Assume there exists a continuous function $\phi:X\to\mathbb{R}$ such
that $\mu(fE)=\int_E e^{\phi(x)}d\mu(x)$ for each measurable subset
$E\subset M$. The function $e^\phi$ is called the {\it Jacobian} of
$f$ with respect to $\mu$. As a byproduct of Theorem \ref{fund}, we
give a proof of a very special case of \cite[Corolary 24]{K} (simple
and regular action) without using Rokhlin disintegration theorem:
\begin{cor}\label{special}
Let $f:X\to X$ be a homeomorphism and $\mu$ be a Borel measure with
Jacobian $J_f=e^\phi$, where $\phi$ is a continuous function on $X$.
Then the dissipative part $D_f$ of $(\mu,f)$ satisfies
$\mu(D_f\Delta H(f,\phi))=0$.
\end{cor}
\begin{proof}
(1). To show $\mu(D_f\backslash H(f,\phi))=0$, it suffices to show
that for each wandering set $W$,
$\sum_{n\in\mathbb{Z}}e^{\phi_n(x)}<+\infty$ for $\mu$-a.e. $x\in
W$. This is true since
\[\int_W \sum_{n\in\mathbb{Z}}e^{\phi_n(x)}dm(x)
=\sum_{n\in\mathbb{Z}}\int_{f^nW}dm(x)=m(\bigcup_{n\in\mathbb{Z}}f^nW)\le
1.\]

(2). Now we show $\mu(H(f,\phi)\backslash D_f)=0$. It is trivial if
$\mu(H(f,\phi))=0$. Then assume $\mu(H(f,\phi))>0$. Let $W$ be the
fundamental domain of $H(f,\phi)$ given by Theorem \ref{fund}. Then
$\mu(W)>0$ and $W$ is wandering. So $H(f,\phi)\subset D_f$.
\end{proof}

\section{Applications: smooth systems}\label{smooth}

In this section we give some estimates about the transitive sets of
partially hyperbolic systems. Let $M$ be a compact Riemannian
manifold without boundary. Recall that $x\in M$ is a transitive
point of $f$ if its orbit $\mathcal{O}_f(x)$ is dense on $M$. Denote
by $\mathrm{Tran}_f$ be the set of transitive points.

A $C^r$ diffeomorphism $f:M\to M$ is said to be {\it partially
hyperbolic} if there are a $Tf$-invariant splitting of $TM=
E^s\oplus E^{c}\oplus E^u$, a smooth Riemannian metric $g$ on $M$
and positive constants $\nu,\tilde{\nu},\gamma$ and $\tilde{\gamma}$
with $\nu,\tilde{\nu}<1$ and
$\nu<\gamma\le\tilde{\gamma}^{-1}<\tilde{\nu}^{-1}$ such that, for
all $x\in M$ and for all unit vectors $v\in E^s_x$, $w\in E^c_x$ and
$v'\in E^u_x$,
\begin{equation}\label{partial}
\|Tf(v)\|\le\nu<\gamma\le\|Tf(w)\|\le\tilde{\gamma}^{-1}
<\tilde{\nu}^{-1}\le\|Tf(v')\|.
\end{equation}
We assume that both $E^s$ and $E^u$ are nontrivial and continuous.
It is well known that $E^s$ and $E^u$ are uniquely integrable and
tangent to the stable foliation $\mathcal{W}^s$ and the unstable
foliation $\mathcal{W}^u$ respectively\footnote{Although these
foliations may not be smooth, they are transversal absolutely
continuous with $C^r$ leaves.}. Let $m$ be the normalized measure
induced by the Riemannian metric $g$ on $M$.

\begin{definition}
Let $A$ be a measurable subset of $M$. Then $A$ is said to be {\it
$s$-saturated} if for each $x\in A$, $W^s(x)\subset A$. Similarly we
can define {\it $u$-saturated} sets. Then the set $A$ is {\it
bi-saturated} if it is $s$-saturated and $u$-saturated.
\end{definition}

The following is slightly general version of above one:
\begin{definition}
Let $A$ be a measurable subset of $M$. Then $A$ is said to be {\it
essentially $s$-saturated} if there exists an $s$-saturated set
$A^s$ with $m(A\Delta A^s)=0$. Similarly we can define {\it
essentially $u$-saturated} sets. The set $A$ is {\it essentially
bi-saturated} if there exists a bi-saturated set $A^{su}$ with
$m(A\Delta A^{su})=0$, and {\it bi-essentially saturated} if $A$ is
essentially $s$-saturated and essentially $u$-saturated.
\end{definition}

It is worth to point out that there is a subtle difference between
essential bi-saturation and bi-essential saturation, see \cite{BW}.

\begin{definition}
A partially hyperbolic diffeomorphism $f: M\to M$ is said to be {\it
accessible} if each nonempty bi-saturated set is the whole manifold
$M$. The map $f$ is {\it essentially accessible} if every measurable
bi-saturated set has either full or zero volume.
\end{definition}
Dolgopyat and Wilkinson proved in \cite{DW} that accessibility holds
on a $C^1$-open and $C^1$-dense subset of partially hyperbolic
systems.

Now we are ready to prove the following proposition:
\begin{pro}\label{biess}
Let $f:M\to M$ be a $C^2$ partially hyperbolic diffeomorphism and
$C_f$ the conservative part of $(m,f)$. Assume $m(C_f)>0$.
\begin{enumerate}
\item Then every $f$-invariant subset $E\subset C_f$ is bi-essentially saturated.

\item Moreover if $f$ is essentially accessible, then $m$-a.e. $x\in
C_f$ is a transitive point. In particular $m(\mathrm{Tran}_f)\ge
m(C_f)>0$.
\end{enumerate}
\end{pro}
Note that $m(C_f)+m(H(f,\log J_f))=m(C_f)+m(D_f)=1$ by Corollary
\ref{special}. So Theorem B follows from this proposition. We need
the following :

\noindent {\bf Halmos Recurrence Theorem} (Theorem 1.1.1 in
\cite{A}). {\it Let $C_f$ be the conservative part of $(m,f)$. Then
for every measurable subset $A\subset C_f$,
$\sum_{n\ge1}1_A(f^nx)=+\infty$ for $m$-a.e. $x\in A$. In other
words, $m$-a.e. $x\in A$ will return to $A$ infinitely many times. }

\begin{proof}[Proof of the first conclusion.]
Let $E\subset C_f$ be an $f$-invariant subset. We first show that
$E$ is essentially $s$-saturated. Note that for each $x\in M$, the
stable manifold $W^s(x)$ is a $C^2$ immersed submanifold. Denote by
$m_{W^s(x)}$ (by $m_s$ for short) the leaf volume induced by the
restricted Riemannian metric on $W^s(x)$.

Note that the proof of Lemma 4.1 in \cite{X} also works for our
case. The only difference is that we use Halmos Recurrence Theorem,
instead of Poincar\'e Recurrence Theorem. So there exists a
measurable subset $A\subset E$ with $m(E\backslash A)=0$ such that
\begin{equation}\label{global}
m_s(W^s(x)\backslash E)=0,\,\text{ for each }\,x\in C.
\end{equation}

To construct an $s$-saturate of $E$, we need the following fact,
which follows from the continuity of the foliation $\mathcal{W}^s$:

\begin{itemize}
\item If $K$ is closed subset of $M$, then
$\bigcup_{x\in K}\overline{W^s_R(x)}$ is closed for all $R>0$.
\end{itemize}

Now let $K_n\subset K_{n+1}\subset\cdots\subset C$ be an increasing
sequence of compact subsets with $m(C\backslash K_n)\to0$. It is
easy to see that the set $C^s=\bigcup_{n\ge1}\bigcup_{x\in
K_n}\overline{W^s_n(x)}$ is measurable, $s$-saturated and
\begin{enumerate}
\item[a)] $m(E\backslash C^s)=0$ since
$m(E\backslash C^s)\le m(E\backslash K_n)= m(C\backslash K_n)\to0$
as $n\to\infty$.

\item[b)]
$m(C^s\backslash E)=0$ by \eqref{global} and by absolute continuity
of $\mathcal{W}^s$.
\end{enumerate}
So $m(E\Delta C^s)=0$ for an $s$-saturated set $C^s$. Therefore $E$
is essentially $s$-saturated. The essential $u$-saturate property of
$E$ follows similarly. This finishes the proof.
\end{proof}

To prove the second conclusion, we first show that for each open
ball $B$, $\mathcal{O}(x)\cap B\neq\emptyset$ for $m$-a.e. point
$x\in C_f$. To the end we consider $G(B)$, the subset of points $x$
which has a neighborhood $U$ of $x$ such that $\mathcal{O}(y)\cap
B\neq\emptyset$ for $m$-a.e. $y\in U\cap C_f$. Evidently $G(B)$ is a
nonempty open subset (and $f$-invariant). Note that we can replace
$C_f$ by its $s$-saturate $C^s$ in the definition of $G(B)$ since
$m(C_f\Delta C^s)=0$.
\begin{lem}
The set $G(B)$ is bi-saturated and $m(G(B))=1$.
\end{lem}

\begin{proof}
Let us prove $G(B)$ is $s$-saturated. It suffices to show that $q\in
G(B)$ for each $q\in W^s_\delta(p)$ and each $p\in G(B)$, where the
size $\delta$ is fixed. So the justification lies in a local
foliation box $Z$ of $\mathcal{W}^s$ around $p$.

For a point $x\in Z$, denote $W^s_Z(x)$ the component of $W^s(x)\cap
Z$ that contains $x$. Since $p\in G(B)$, there exists a small
neighborhood $U$ of $p$ with $\mathcal{O}(y)\cap B\neq\emptyset$ for
$m$-a.e. $y\in U\cap C^s$.

Let $R$ be the set of recurrent points in $U\cap C^s$ whose orbits
enter $B$. Clearly $m(U\cap C^s\backslash R)\le m(C_f\backslash
R)=0$. So we can pick a smooth transverse $\tau\subset U$ of
$\mathcal{W}^s_{Z}$ such that $\tau\cap W^s_U(p)\neq\emptyset$ and
$m_\tau(C^s\backslash R)=0$, where $m_\tau$ is the induced volume on
$\tau$ (note that $C^s$ is not only essentially $s$-saturated, but
$s$-saturated). Now we have the set $\bigcup_{x\in \tau\cap
R}W^s_Z(x)$ has full $m$-measure in the set $\bigcup_{x\in \tau\cap
C^s}W^s_Z(x)=\left(\bigcup_{x\in \tau}W^s_Z(x)\right)\cap C^s$.

The set $\bigcup_{x\in \tau}W^s_Z(x)$ contains an open neighborhood
$V$ of $q$. Moreover $\mathcal{O}(y)\cap B\neq\emptyset$ for
$m$-a.e. $y\in V\cap C^s$ and therefore $q\in G(B)$. This implies
$G(B)$ is $s$-saturated. Similarly $G(B)$ is $u$-saturated and hence
$m(G(B))=1$ by the essential accessibility of $f$.
\end{proof}
The rest of the proof follows closely from the proof of Theorem 5.5
in \cite{Z} and hence is omitted here. So $m$-a.e. $x\in C_f$ is a
transitive point and $m(\mathrm{Tran}_f)\ge m(C_f)>0$. This
completes the proof of Proposition \ref{biess}.

\vskip.2cm

To get sharper results we need the following definition:
\begin{definition}\label{centerb}
A partially hyperbolic diffeomorphism $f$ is {\it center bunched} if
the constants $\nu$, $\tilde{\nu}$ and $\gamma$, $\tilde{\gamma}$
given in \eqref{partial} can be chosen so that:
$\nu<\gamma\tilde{\gamma}$ and $\tilde{\nu}<\gamma\tilde{\gamma}$.
\end{definition}
\begin{pro}[Corollary 5.2 in \cite{BW}]\label{BW}
Let $f:M\to M$ be a $C^2$ center bunched partially hyperbolic
diffeomorphism. Then every measurable bi-essentially saturated
subset is essentially bi-saturated.
\end{pro}

The following is a direct corollary of Proposition \ref{biess} and
\cite[Corollary 5.2]{BW}, which provides a partial generalization of
Gurevich and Oseledets dichotomy:
\begin{cor}
Let $f:M\to M$ be a $C^2$ essentially accessible, center bunched
partially hyperbolic diffeomorphism.
\begin{enumerate}
\item Either $D_f=M$: then $(m,f)$ is completely dissipative,

\item or $C_f=M$: then $(m,f)$ is ergodic.
\end{enumerate}
\end{cor}
Recall that $(m,f)$ is said to be {\it ergodic} if every measurable,
$f$-invariant subset $E$ satisfies $m(E)=0$ or $1$. Note that $m$
may not be $f$-invariant.
\begin{proof}
Assume $m(C_f)>0$. Then $C_f$ is bi-essentially saturated and hence
also essentially bi-saturated by Proposition \ref{BW}. So $m(C_f)=1$
by the essential accessibility of $f$. Hence $D_f=\emptyset$ and
$C_f=M$.

Now let $E\subset M$ be an $f$-invariant subset with $m(E)>0$. Since
$C_f=M$, we get that $E$ is also bi-essentially saturated by
Proposition \ref{biess}. So $E$ is essentially bi-saturated by
Proposition \ref{BW} and $m(E)=1$ by the essential accessibility of
$f$. This shows that $(m,f)$ is ergodic.
\end{proof}

\section{Largeness of Heteroclinic sets}\label{entropy}

In this section we prove that the heteroclinic sets can have large
entropy. We first give the definition of the Bowen dimensional
entropy $h_B(f,\cdot)$ for noncompact subsets \cite{Bo} with respect
to a homeomorphism $f:X\to X$. For $k\ge1$ and $x,y\in X$, let
$d_k(x,y)=\max\{ d(f^ix,f^iy):0\le i<k\}$, and let $B(x,r,k)=\{y\in
X: d_k(x,y)<r\}$ be the {\it Bowen ball} of radius $r>0$.

Let $E\subseteq X$ and $t\ge 0$. For any $\epsilon>0$ and $n\ge1$,
denote
$$M_n(f,E,t,\epsilon)=\inf \{ \sum_{i\ge1} e^{-tn_i}:
\bigcup_{i\ge1} B(x_i,r,n_i) \supseteq E \mbox{ and } n_i\ge n
\mbox{ for each }i\ge1 \}.$$ Since $M_n(f,E,t,\epsilon)$ is
increasing with respect to $n\in \mathbb{N}$, the limit
$$M(f,E,t,\epsilon)=:\lim_{n\rightarrow\infty}M_n(f,E,t,\epsilon)$$
is well defined. It is clear that $M(f,E,t,\epsilon)\le
M(f,E,s,\epsilon)$ if $t\ge s\ge 0$ and $M(f,E,t,\epsilon)\notin
\{0, +\infty\}$ for at most one point $t\ge 0$. Then define
\begin{equation}
h_B(T,E,\epsilon)=\inf \{t\ge 0: M(f,E,t,\epsilon)= 0\}=\sup \{t\ge
0: M(f,E,t,\epsilon)=+\infty\}.
\end{equation}
The Bowen dimensional entropy of $E$ is
$h_B(f,E)=\lim_{\epsilon\to0}h_B(T,E,\epsilon)$. Note that
\begin{equation}\label{subadd}
h_B(f,\bigcup_{i\ge1}E_i,\epsilon)=\max_{i\ge1} h_B(f,E_i,\epsilon),
\text{ and } h_B(f,\bigcup_{i\ge1}E_i)=\max_{i\ge1} h_B(f,E_i).
\end{equation}

Let $f:M\to M$ be a transitive Anosov diffeomorphism and
$P=[W^s_{loc}(x),W^u_{loc}(x)]$ be a small rectangle. Let $y,z\in
W^s_{loc}(x)$ and $h^s_{y,z}:W^u_{P}(y)\to W^u_P(z)$ be the local
stable holonomy (homeomorphism) with respect to $W^s$.
\begin{lem}\label{holonomy}
Let $E\subset W^u_P(x)$. Then $h_B(f,h^s_{x,y}E,2\epsilon)\le
h_B(f,E,\epsilon)\le h_B(f,h^s_{x,y}E,\epsilon/2)$ for all
$\epsilon>0$ and for all $y\in W^s_{loc}(x)$. In particular
$h_B(f,h^s_{x,y}E)=h_B(f,E)$ for all $y\in W^s_{loc}(x)$.
\end{lem}
\begin{proof}
Step 1. Let $\epsilon>0$. Note that $fB(z,\epsilon,n+1)\subset
B(fz,\epsilon,n)$ always holds. So $M_n(f,fE,t,\epsilon)\le e\cdot
M_{n+1}(f,E,t,\epsilon)$ and hence $h_B(f,fE,\epsilon)\le
h_B(f,E,\epsilon)$. Now we show the other direction.

Let $n\ge1$ and $\{B(z_l,\epsilon,n_l):n_l\ge n\}$ be a covering of
$fE$. Then $B(f^{-1}z_l,\epsilon)\cap E\supset
f^{-1}(B(z_l,\epsilon)\cap fE)$ (since $f^{-1}|_{fW^u_P(x)}$ is
contracting) and hence
\begin{align*}
B(f^{-1}z_l,\epsilon,n_l+1)\cap E& =B(f^{-1}z_l,\epsilon)\cap
f^{-1}B(z_l,\epsilon,n_l)\cap E \\&\supset
f^{-1}(B(z_l,\epsilon)\cap fE)\cap f^{-1}B(z_l,\epsilon,n_l)\supset
f^{-1}(B(z_l,\epsilon,n_l)\cap fE).
\end{align*}
So $\{B(f^{-1}z_l,\epsilon,n_l+1):n_l\ge n\}$ be a covering of $E$.
This implies that $M_n(f,fE,t,\epsilon)\ge e\cdot
M_{n+1}(f,E,t,\epsilon)$ and $h_B(f,fE,\epsilon)\ge
h_B(f,E,\epsilon)$. So $h_B(f,fE,\epsilon)= h_B(f,E,\epsilon)$ for
all $\epsilon>0$. Breaking $f^kE$ into small pieces, applying
\eqref{subadd} and using inductive argument, we see
$h_B(f,f^kE,\epsilon)= h_B(f,E,\epsilon)$ all $k\ge1$.

Step 2. Pick $\delta>0$ such that $d(h^s_{x,y}(p),p)<\delta$ for all
$p\in W^u_{P}(x)$. Iterating forward by $f^k$, we get a new
homeomorphism $h^s_k:f^kW^u_{P}(x)\to f^kW^u_P(y)$ induced by
$f^k\circ h^s_{x,y}\circ f^{-k}$. Note that
$d(h^s_k(f^kp),f^kp)<\lambda^k\delta$ for all $p\in W^u_{P}(x)$,
where $\lambda\in(0,1)$ is the contraction rate on $W^s$. Also note
that $f^k\circ h^s_{x,y}E=h^s_k\circ f^kE$.

Step 3. Let $E\subset W^u_P(x)$, $\epsilon>0$ and
$t>h=h_B(f,E,\epsilon)$. By Step 1 we see $h_B(f,f^kE,\epsilon)<t$
for all $k\ge1$.

Pick $N=N_{\epsilon,k}\ge1$ large such that for each $n\ge N$, there
exists a countable cover of $f^kE$, say $\{B(z_l,\epsilon,n_l):
z_l\in f^kE, n_l\ge n\}$, such that $\sum_{l\ge1}e^{-t
n_l}<M_n(f,f^kE,t,\epsilon)+1.$

For each $q\in f^kh^s_{x,y}E$, we know $q=h^s_k(p)\in h^s_kf^kE$
where $p\in f^kE\cap B(z_l,\epsilon,n_l)$ for some $l\ge1$. So
\[d_n(h^s_kz_l,q)\le d_n(h^s_kz_l,z_l)+
d_n(z_l,p)+d_n(p,q)\le 2\lambda^k\delta+\epsilon.\] Therefore the
collection $\{B(h^s_kz_l,2\lambda^k\delta+\epsilon,n_l): n_l\ge n\}$
forms a cover $f^kh^s_{x,y}E$ and
$$M_n(f,f^kh^s_{x,y}E,t,2\lambda^k\delta+\epsilon)\le\sum_{l\ge1}e^{-t
n_l}< M_n(f,f^kE,t,\epsilon)+1.$$ Passing $n$ to infinity, we see
$M(f,f^kh^s_{x,y}E,t,2\lambda^k\delta+\epsilon)\le 1$. Hence for
every every $k\ge1$,
\begin{itemize}
\item $h_B(f,h^s_{x,y}E,2\lambda^k\delta+\epsilon)=
h_B(f,f^kh^s_{x,y}E,2\lambda^k\delta+\epsilon)\le t$.
\end{itemize}
Picking $k$ large with $2\lambda^k\delta<\epsilon$, we see
$h_B(f,h^s_{x,y}E,2\epsilon)\le t$. Then passing $t$ to $h$, we see
$h_B(f,h^s_{x,y}E,2\epsilon)\le h=h_B(f,E,\epsilon)$. Note that
$h^s_{y,x}\circ h^s_{x,y}=Id$. So we can prove $h_B(f,E,\epsilon)\le
h_B(f,h^s_{x,y}E,\epsilon/2)$ for all $\epsilon>0$. Finally passing
$\epsilon$ to zero and applying {\it Squeeze Theorem}, we see
$h_B(f,h^s_{x,y}E)=h_B(f,E)$ for every $y\in W^s_{loc}(x)$. This
completes the proof.
\end{proof}

\begin{lem}\label{rectangle}
Let $P=[W^s_{loc}(x),W^u_{loc}(x)]$ be a rectangle and $E\subset P$
be $W^s_P$-saturated, then $$h_B(f,E\cap W^u_P(y),\epsilon/2)\le
h_B(f,E,\epsilon)\le h_B(f,E\cap W^u_P(y),\epsilon/2)$$ for every
$\epsilon>0$ and for every $y\in W^s_{loc}(x)$. In particular
$h_B(f,E)=h_B(f,E\cap W^u_P(y))$ for every $y\in W^s_{loc}(x)$
\end{lem}
\begin{proof}
Let $\epsilon>0$ and $y\in W^s_{loc}(x)$ be given. Clearly $E\cap
W^u_P(z)=h^s_{y,z}(E\cap W^u_P(z))$. Then by Lemma \ref{holonomy} we
see $h_B(f,E\cap W^u_P(z),\epsilon)\le h=h_B(f,E\cap
W^u_P(y),\epsilon/2)$ for every $z\in W^s_{loc}(x)$. Let $t>h$ be
fixed. Pick a $\delta$-dense subset $\{x_1,\cdots,x_d\}\subset
W^s_P(x)$. Then $h_B(f,E\cap W^u_P(x_j),\epsilon)<t$ for all
$j=1,\cdots,d$.

There exists $N=N_{\epsilon,d}\ge1$ such that for each $n\ge N$,
there exists a countable cover of $E\cap W^u_P(x_j)$, say
$\mathcal{C}_{j,n}=\{B(z_l^j,\epsilon,n_l^j): z_l^j\in E\cap
W^u_P(x_j), n_l^j\ge n\}$, such that
$$\sum_{l\ge1}e^{-t n_l^j}<M_n(f,E\cap
W^u_P(x_j),t,\epsilon)+\frac{1}{2d}<\frac{1}{d}.$$ For each $y\in
W^s_\epsilon(x_j)$, denote the stable holonomy by
$h^s_{j,y}:W^u_P(x_j)\to W^u_P(y)$ and $s(\delta)=d(h^s_{j,y},Id)$.
For each $q\in E\cap W^u_P(y)$ there exists $p\in E\cap W^u_P(x_j)$
with $h^s_{j,y}(p)=q$. Also $p\in B(z_l^j,\epsilon,n_l)$ for some
$l\ge1$. So
\[d_n(z_l^j,q)\le d_n(z_l^j,p)+d_n(p,q)\le s(\delta)+\epsilon.\]
So $\{B(z_l^j,s(\delta)+\epsilon,n_l^j):l\ge1\}$ covers $E\cap
W^u_P(W^s_\epsilon(x_j))$ for each $j=1,\cdots,d$. Therefore
$\{B(z_l^j,s(\delta)+\epsilon,n_l^j):l\ge1,1\le j\le d\}$ covers $E$
and
\[M_n(f,E,t,s(\delta)+\epsilon)\le
\sum_{j=1}^d\sum_{l\ge1}e^{-t n^j_l}\le\sum_{j=1}^d 1/d=1.\]

Passing $n$ to infinity, we see $M(f,E,t,s(\delta)+\epsilon)\le 1$
and hence $h_B(f,E,s(\delta)+\epsilon)\le t$ for every $\delta>0$.

Passing $\delta$ to zero, we see $s(\delta)<\epsilon$ and
$h_B(f,E,2\epsilon)\le t$. Then passing $t$ to $h$, we see
$h_B(f,E,2\epsilon)\le h=h_B(f,E\cap W^u_P(y),\epsilon/2)$. Clearly
$h_B(f,E,2\epsilon)\ge h_B(f,E\cap W^u_P(y),2\epsilon)$. Finally
passing $\epsilon$ to zero and applying Squeeze Theorem we get
$h_B(f,E)=h_B(f,E\cap W^u_P(y))$ for every $y\in W^s_{loc}(x)$. This
finishes the proof.
\end{proof}

\vskip.3cm

Now we give the proof of Proposition D.
\begin{proof}[Proof of Proposition D]
Let $f:M\to M$ be a transitive Anosov diffeomorphism. Note that $f$
satisfies {\it Specification Property}. Pfister and Sullivan proved
in \cite{PS} that $h_B(f,B(\mu,f))=h_\mu(f)$ for every $\mu\in
\mathcal{M}(f)$ (by Proposition 2.1 and Theorem 1.2 in there). Then
for the heteroclinic set $H_f(\mu,\nu)=B(\mu,f)\cap B(\nu,f^{-1})$,
\begin{equation}\label{lower}
h_B(f,H_f(\mu,\nu))\le h_B(f,B(\mu,f))=h_\mu(f).
\end{equation}

Then we cover $M$ by rectangles, say $\{P_1,\cdots,P_d\}$ and assume
$h_B(f,B(\mu,f)\cap P_i)=h_B(f,B(\mu,f))$ for some
$P_i=[W^s_{loc}(x_i),W^u_{loc}(x_i)]$. Note that $E=B(\mu,f)\cap
P_i$ is $W^s_{P_i}$-saturated. By Lemma \ref{rectangle} we see
$h_B(f,B(\mu,f)\cap P_i)=h_B(f,B(\mu,f)\cap W^u_{P_i}(y))$ for every
$y\in W^s_{loc}(x_i)$.

As observed in Remark \ref{large}, we know that $W^u_{P_i}(y)\subset
B(\nu,f^{-1})$ for some $y\in W^s_{loc}(x_i)$, since $B(\nu,f^{-1})$
is dense and $u$-saturated. So $B(\mu,f)\cap W^u_{P_i}(y)\subset
B(\mu,f)\cap B(\nu,f^{-1})=H_f(\mu,\nu)$. Then we have

\begin{align}\label{upper}
h_B(f,H_f(\mu,\nu))&\ge h_B(f,B(\mu,f)\cap W^u_{P_i}(y))\\
&=h_B(f,B(\mu,f)\cap P_i)=h_B(f,B(\mu,f)\cap P_i)=h_\mu(f).\nonumber
\end{align}
Combing \eqref{lower} and \eqref{upper}, we see
$h_B(f,H_f(\mu,\nu))=h_\mu(f)$.

Note that $H_{f^{-1}}(\nu,\mu)=H_f(\mu,\nu)$ and
$h_\nu(f)=h_\nu(f^{-1})$. Therefore
\[h_B(f^{-1},H_f(\mu,\nu))=h_B(f^{-1},H_{f^{-1}}(\nu,\mu))=h_\nu(f^{-1})=h_\nu(f).\]

So if $h_\mu(f)\neq h_\nu(f)$ for some $\mu,\nu\in\mathcal{M}(f)$,
then $h_B(f,H_f(\mu,\nu))\neq h_B(f^{-1},H_f(\mu,\nu))$. This
completes the proof.
\end{proof}

\end{document}